\documentclass[10pt]{amsart}
\usepackage{amsmath, amssymb}
\usepackage{graphicx}
\usepackage{subfigure}
 \usepackage{mathrsfs}

\usepackage{color}
\usepackage[english,american]{babel}

\usepackage{esint}
\usepackage[svgnames]{xcolor} 
\usepackage{dsfont}
\usepackage{url}
\usepackage[utf8]{inputenc}
\usepackage[T1]{fontenc}
\usepackage{lmodern}
\usepackage{babel}
\usepackage{mathtools}  
\usepackage{amssymb}
\usepackage{lipsum}
\usepackage{mathrsfs}
\usepackage{color}

\usepackage{epsf,graphicx,epsfig,cite,latexsym}
\usepackage[export]{adjustbox}
\usepackage{mathtools}
\usepackage{latexsym, amsmath, amssymb, a4, epsfig}
\usepackage[T1]{fontenc}

\usepackage{listings}

\usepackage{algorithm}
\usepackage{algpseudocode}

 \usepackage{mathrsfs}
\newcommand{\no}[1]{#1}
\renewcommand{\no}[1]{}
\no{\usepackage{times}\usepackage[subscriptcorrection, slantedGreek, nofontinfo]{mtpro}
\renewcommand{\Delta}{\upDelta}}
\usepackage{color}

\numberwithin{algorithm}{section}
\numberwithin{figure}{section}

\newtheorem{lemma}{Lemma}[section]
\newtheorem{remark}{Remark}[section]

\newtheorem{proposition}{Proposition}[section]
\newtheorem{theorem}{Theorem}[section]

\date{\today}

\newcommand{\C}{\mathbb{C}}

\let\conjugatet\overline 
\def\cc{\conjugatet}
\def\w{\widetilde }
\def\ep{\varepsilon}

\def\t{\tau}
\def\b{\bigg}
\newcommand{\p}{\partial}
\def\no{\nonumber}
\newcommand{\norm}[1]{\|#1\|}
\newcommand{\abs}[1]{\left|#1\right|}

\def\qed{\ifmmode\hbox{\hfill\sqb}\else{\ifhmode\unskip\fi%
\nobreak\hfil
\penalty50\hskip1em\null\nobreak\hfil\sqb
\parfillskip=0pt\finalhyphendemerits=0\endgraf}\fi}
\newcommand{\intt}{\int^t_0}

\newcommand{\paxi}{\partial_{\xi}}

\newcommand{\ds}{\displaystyle}

\newcommand{\be}{\begin{equation}}
\newcommand{\ee}{\end{equation}}
\newcommand{\ba}{\begin{array}}
\newcommand{\ea}{\end{array}}
\newcommand{\bea}{\begin{eqnarray*}}
\newcommand{\eea}{\end{eqnarray*}}
\newcommand{\bean}{\begin{eqnarray}}
\newcommand{\eean}{\end{eqnarray}}
\def\im{\textrm{Im}}
\def\re{\textrm{Re}}

\def\sqw{\hbox{\rlap{\leavevmode\raise.3ex\hbox{$\sqcap$}}$%
\sqcup$}}
\def\sqb{\hbox{\hskip5pt\vrule width4pt height6pt depth1.5pt%
\hskip1pt}}

\begin{document}

\title[]{{ Identification of  a boundary influx condition in a one-phase Stefan problem }}

\author{Chifaa Ghanmi}
\address{Chifaa Ghanmi, Facult\'e des Sciences de Tunis, Universit\'e Tunis El Manar, Tunisia}
\email{Chifaa.Ghanmi@fst.utm.tn}

\author{Saloua Mani-Aouadi}
\address{Saloua Mani Aouadi,  Facult\'e des Sciences de Tunis, Universit\'e Tunis El Manar, Tunisia}
\email{Saloua.Mani@fst.utm.tn}

\author{Faouzi Triki}

\address{Faouzi Triki,  Laboratoire Jean Kuntzmann,  UMR CNRS 5224, 
Universit\'e  Grenoble-Alpes, 700 Avenue Centrale,
38401 Saint-Martin-d'H\`eres, France}

\email{faouzi.triki@univ.grenoble-alpes.fr}

\thanks{
The work of FT is supported in part by the
 grant ANR-17-CE40-0029 of the French National Research Agency ANR (project MultiOnde).}

\begin{abstract}
We consider a one-dimensional one-phase  inverse Stefan problem for the heat equation. 
It consists in recovering a boundary influx condition  from the knowledge  of 
 the position of the moving  front, and the initial state. We derived  a logarithmic  stability estimate  that shows that the inversion
  is ill-posed. The proof is based on  integral equations and  unique continuation for holomorphic functions.
   We also  proposed a direct algorithm with a regularization term to 
  solve the nonlinear  inverse problem.  Several numerical tests using noisy data are provided  with relative errors 
  analysis.
   \end{abstract}

\maketitle

\section{Introduction}
Stefan problem is a specific type of free boundary problems in partial differential equations related to heat diffusion.  It aims to 
describe the temperature distribution in a homogeneous  medium  undergoing a phase change, for example water passing to ice. 
Stefan problems model finds application in many  engineering  settings in which there is melting or  freezing  causing a boundary 
change in time. Examples include melting of ice, recrystallization of metals,  tumor growth,  freezing of liquids, etc \cite{Ru, Go}. \\

The Stefan problem for the heat equation  consists in  determining the temperature and 
 location of the melting front delimiting the different phases when the  initial and boundary conditions are given. Conversely, the inverse Stefan problem is to recover boundary conditions, and/or initial condition  from   measurement of the moving boundary position. The Stefan problem as a mathematical model  has been studied in  the literature for many decades, and has attracted the attention of many mathematicians and physicists (see for instance \cite{Ru, Me, Vi} and references therein). \\
 
 The direct Stefan problem is known to be well-posed, i.e., it possesses a unique solution which depends continuously on the data, provided that the initial state and the source functions have the correct signs \cite{Fr, LSU}. If the sign requirements are not fulfilled,  the direct  Stefan problem may not have a global solution or may  be an ill-posed problem \cite{DF, PSS}.  In contrast with the direct Stefan problem only few theoretical results are available for the related inverse problem.  Indeed most published materials have considered the  numerical reconstruction of the  temperature or heat flux  on the boundary \cite{RK, Jo1, Jo2, Co, johansson2011method, Wr}. Solving numerically  both direct and inverse Stefan problems can be a difficult task because of the free boundary,   the nonlinearity and  the  ill-posedness  in the sense that these problems  may not possess a solution, or that in case there exist  solutions they  may not depend continuously on the given data  \cite{cannon1967stability, WY, RK}.    \\

In this paper we are interested in solving the so-called one-dimensional one-phase inverse Stefan problem which is {\it to find  the heat influx at the boundary from the knowledge of the moving free boundary.} The one-dimensional Stefan problem has been extensively studied by many authors \cite{Ru, Fr2, cannon1984one, WY, johansson2011method, Sa}.  Assuming  that  the temperature of the solid is constant, and  the heat conduction in the melted portion is governed by the heat equation, leads to a  simple one-phase free boundary value problem. If in addition the heat influx stays positive, which means that heat is continuously entering the domain, only one phase remains,  and  the direct Stefan problem admits  a unique global smooth solution. The inverse problem can be recasted  as  a non standard Cauchy problem. The uniqueness of the inversion has been established by several  authors under different smoothness assumptions on the boundary influx and  the initial state \cite{Ky, cannon1967cauchy, CH, Fr2}.  The novelty in this paper is a logarithmic stability  estimate of the inversion, and the numerical treatment  of the problem  based on the resolution of  regularized linear integral equations. \\

This paper is composed of  five sections.  Section 2 provides an introduction to the direct problem.  We first recall  known useful existence, uniqueness, stability and regularity properties for solutions of the classical one  Stefan problem in Theorems~\ref{thm 1.19} and \ref{stab}. We 
set the inverse Stefan problem in section 3. The main stability estimate is given in Theorem \ref{main theorem}. We detailed the proof of the 
principal  results in  section  4. Finally, section 5 is devoted to the numerical resolution of  both the direct and inverse Stefan problems. Several numerical tests are provided  with relative errors analysis.
 
\section{The direct problem}

We next introduce the one dimension one-phase Stefan problem. Let $b>0$  and $T>0$ be  two fixed constants.  
For any positive function $s\in C([0,T])$ satisfying  $s(0)=b$, we define the open set  $Q_{s,T} \subset \mathbb R_+\times(0, T)$,
by
\begin{equation}
 Q_{s,T}=\{ (x,t) |\ \  0<x<s(t), \quad 0<t<T \}.
\end{equation}
The direct Stefan problem is to determine 
$u(x,t)\in C(\cc {Q_{s,T}})\cap C^{2,1} (Q_{s,T})$ and $s(t) \in C^1((0,T])\cap C^{0, 1}([0,T])$, satisfying 
 \begin{align}
    u_t-u_{xx}&=0  &\quad in \quad  Q_{s,T}, \label{1.16} \\
   -u_x(0,t)&=h(t)>0 \quad &0<t<T, \label{1.17} \\
   u(x,0)&=u_0(x) \geq 0 \quad &0<x<b, \label{1.18}\\
   -u_x(s(t),t)&=\overset{.}{s}(t) \quad &0<t<T, \label{1.19}\\   
   u(s(t),t)&=0 \quad &0<t<T, \label{1.20}\\
   s(0)&=b, \label{1.21}
  \end{align}
where $h\in C([0,T])$, $u_0\in C([0,b])$ are given.\\

Let $H>0$ be a fixed constant. Throughout  the paper we assume 
   \begin{equation}\label{1.22}
 h\in C([0,T]), \ \ h(t)>0 \ \ 0\leq t\leq T,
\end{equation}
\begin{equation}\label{1.23}
 u_0\in C([0,b]), \ \ 0\leq u_0(x)\leq H(b-x), \ \ 0\leq x\leq b.
\end{equation}
Let
\bean \label{M}
M= \max\{\|h\|_\infty, \, H \}.
\eean

Next, we give a result of existence and uniqueness of the Stefan problem  (\ref{1.16})-(\ref{1.21}).

\begin{theorem}\label{thm 1.19} 
The problem (\ref{1.16})-(\ref{1.21}) admits a unique solution $(s(t), u(x,t))$. In addition the solution  $(s(t), u(x,t))$
satisfies
\bean \label{ineq1}
&0<u(x,t)\leq M(s(t)-x)  &\textrm{  in } Q_{s,T},\\
&0\leq  - u_x(s(t),t)= \overset{.}{s}(t) \leq M & \textrm{  in }  (0, T]. \label{ineq2}
\eean

\end{theorem}

The proof  is based on the Maximum principle and the  fixed point Theorem. The existence and uniqueness of the Stefan problem
has been studied by various authors  under different smoothness assumptions (see for instance \cite{CM, cannon1967stability} and 
references therein). 
\begin{remark} 
The assumption on the sign of the flux $h(t)$ means that heat is entering the domain, so that, for each $t$  only one phase remains.
If the sign requirement on $h(t)$ are not fulfilled,  the direct  Stefan problem may not have a global solution 
or may  be an ill-posed problem \cite{DF, PSS}.
\end{remark}
\begin{remark} \label{sbounds}
Since  $C^{0, 1}([0, T])$ and  $ W^{1, \infty}(0, T) $  are somehow identical for smooth domains 
(Theorem 4, page 294 in  \cite{Ev}), we deduce  from Theorem \ref{thm 1.19} that $s \in W^{1, \infty}(0, T) $. 
Moreover, a simple calculation gives  
\bean \label{sint}
s(t) = b +\int_0^t h(\tau) d\tau + \int_0^b u_0(x) dx- \int_0^{s(t)} u(x, t) dx.
\eean
Since $u$ is nonnegative, we deduce from \eqref{sint}, the following estimate
\bea
b \leq s(t) \leq T\|h\|_{\infty} +b (\|u_0\|_{\infty}+1) & \textrm{  in }  (0, T],
\eea
which combined to inequality \eqref{ineq2}, leads to 
\bean \label{sbound}
\|s\|_{W^{1, \infty}} \leq s_\infty:=\max(T\|h\|_{\infty} +b (\|u_0\|_{\infty}+1); M).
\eean
 \end{remark}

Next we give some useful  properties of the solutions of Stefan problem as well as 
their stability with respect to boundary and  initial data.\\

Consider two sets $(h_i(t), u_{0, i}(x)), i=1, 2 $,  of Stefan data satisfying assumptions \eqref{1.22} and \eqref{1.23}.
Theorem \ref{thm 1.19}  guarantees  the existence of an unique solution $(s_i, u_i)$ to each one of the two problems. 

\begin{theorem} \label{stab}
If $b_1<b_2$, then the free boundaries $s_1(t)$, $s_2(t), $ corresponding to the data $(h_i(t), u_{0, i}(x)), i=1, 2 $, satisfy 
\begin{align}
\abs{s_1(t)-s_2(t)}\leq C\bigg( b_2-b_1+\int^{b_1}_0 \abs{u_{0,1}(x)-u_{0,2}(x)}dx +\int^{b_1}_{b_2} u_{0,2}(x) dx +\intt \abs{h_1(\t)-h_2(\t)}d\t  \bigg), \; 0\leq t\leq T,
\end{align}
where $C>0$ only depends on $T$ and $M$.
\end{theorem}
The proof of this theorem is detailed in \cite{cannon1967stability}.  The constant $C>0$ appearing  in Theorem  \ref{stab}
 is in fact exponentially increasing as function of  $T$ and $M$.

\section{The inverse problem}
In this paper we are interested in the following  inverse problem: {\it to determine  $h=-u_x$ from the knowledge of 
the moving boundary $s$ and  the initial state $u_0$.} Clearly, the inverse Stefan problem is a Cauchy-like problem 
of determining a function $u$ which satisfies \cite{cannon1967cauchy}

\begin{eqnarray}\label{Inverse problem}
(\mathcal{P})\left\{ \begin{array}{rll}
   u_t-u_{xx}&=0  &in \quad  Q_{s,T},\cr
   u(x,0)&=u_0(x) \geq 0,  &0<x<b,\cr
   -u_x(s(t),t)&=\overset{.}{s}(t), &0<t<T,\cr   
   u(s(t),t)&=0, &0<t<T,\cr
   s(0)&=b. & \cr 
\end{array}\right.
\end{eqnarray}
We first  study the uniqueness of recovery of the heat flux for a given free boundary $s\in  C^{1}([0,T])$, and 
a nonnegative initial state $u_0\in C^1([0,b])$. Showing the  uniqueness  in Stefan inverse problem is then 
equivalent to proving  that the linear non-standard Cauchy problem 
$ (\mathcal{P})$ admits a unique solution. The following  result has been  established  by 
J.R. Cannon and J. Douglas in \cite{cannon1967cauchy}.
\begin{proposition} \label{uniqueness}
For given $s\in C^1([0,T])$ and $u_0\in C^1([0,b])$, there can exist at most one solution 
$u\in C^2 (Q_{s,T}) \cap C^1(\cc{Q_{s,T}})$ of the problem $(\mathcal{P})$.
\end{proposition}
Since $u \in C^1(\cc{Q_{s,T}})$, we deduce from Proposition  \ref{uniqueness} that there exists a unique solution 
to inverse  Stefan 
problem given by $h=-u_x$.\\

Notice that required $C^1$ regularity of the free boundary $s(t)$ in the previous uniqueness result  
 is not  guaranteed by the regularity assumptions  \eqref{1.22} and \eqref{1.23} (does not 
 need to be $C^1$ close to $0$). In order to reach 
 such a regularity one needs to impose stronger smoothness  and compatibility conditions on the data
  \cite{Ky, CM}.\\ 
  
Let $A : L^2(0, T) \rightarrow L^2(0, T)$ be the Laplacian $ Ah = -h_{tt} $ with Dirichlet boundary 
condition. It is an unbounded self-adjoint, strictly  positive operator with a compact resolvent. \\

Denote by 
${ D}(A^{\frac{1}{2}})$ the domain of 
$A^{\frac{1}{2}}$, and  introduce  for $\beta \in \mathbb R$ the scale of Hilbert
spaces $X_{\beta}$, as follows: for every
$\beta \geq 0$, $X_{\beta}= {D}(A^{\frac{\beta}{2}})$, with the norm
$\|h \|_\beta=\|A^{\frac{\beta}{2}} h\|_{L^2}$ (note that 
$0 \notin \sigma(A) $ where $\sigma(A)$
is the spectrum of $A$). The space $X_{-\beta}$ is
defined by duality with respect to the pivot space $X$ as
follows: $X_{-\beta} =X_{\beta}^*$ for $\beta>0$ \cite{AT}.
The norm  $\|h \|_\beta$ is given  in terms of the eigenvalues
and eigenfunctions of the operator $A$ \cite{RY}.\\

Let $A_0$ be  the Abel linear integral  operator \cite{Br}

\bea
\label{A0extb}
A_0:  L^2(0, T) \rightarrow L^2(0, T)\\
A_0 h(t) := \frac{1}{\sqrt{\pi}} \int^{t}_0 \frac{1}{\sqrt{t-\t}} h(\t) d\t.
\eea
In \cite{RY}, the authors  showed the equivalence between  the norm $\|h \|_{-\frac{1}{2}}$
and $\|A_0 h\|_{L^2}$, that is, there exists an universal constant $C>0$, such that 
\bean \label{equiv}
C^{-1}\|h \|_{-\frac{1}{2}} \leq \| A_0 h\|_{L^2} \leq C\|h \|_{-\frac{1}{2}}.
\eean

Next, we state  our main result in this paper. 
\begin{theorem}\label{main theorem}  Let
$u_0 \in C([0, T])$ be a given function verifying assumption \eqref{1.23},
 $h, \w h \in C([0, T])$ be strictly positive functions satisfying $\|h\|_{\infty}, \|\w h\|_\infty  \leq M$.
 Let $u$ and $\tilde u$ be the solutions  to  the system (\ref{1.16})-(\ref{1.21}) associated to respectively 
 $h$ and $ \w h$. Denote $s$,    $\w s $  the free boundaries of respectively the solutions $u$ and $\w u$, and
assume that 
\bea
\norm{s-\w {s}}_{W^{1, \infty}} < e^{-1}.
\eea
 Then 
\begin{equation} \label{mainineq}
\|h-\w h\|_{-\frac{1}{2}}\leq  \frac{C}{\ln\left(  \abs{\ln \norm{s-\w s}_{W^{1, \infty}}}\right)},
\end{equation}
where  the constant $C>0$,  depends only on $u_0$, $b$, $T$, $H$,  and  $M$.
\end{theorem}
\begin{remark}
The double logarithmic  stability estimate \eqref{mainineq} shows that the inverse  Stefan problem is indeed  ill-posed. 
The obtained result is then  consistent  with  known stability estimates in  the standard  Cauchy problem for parabolic equations
 in time independent  domains   (see for instance Theorem 1.1 in \cite{CY}). Therefore the motion of the  free boundary causing the 
nonlinearity and  higher complexity of  the direct problem,  does not seem to modify the nature of the Cauchy inversion.

\end{remark}

\begin{remark}
Our stability estimate is  carried out within the Hilbert scale $X_{-\frac{1}{2}}$, but  it  can be also done using 
more  classical spaces. The  characterization  of  $X_{\beta}, \beta \in \mathbb R$,  in terms of Sobolev spaces, can be found 
in \cite{RY}.  The inequality \eqref{equiv} shows that the Abel linear integral  operator  is  invertible from $L^2(0, T)$ to  $X_{-\frac{1}{2}}$,
and hence expressing $h-\w h$  in this later space, is indeed  the best choice when  $L^2$ norm  estimates  of $A_0(h-\w h)$
are available.

\end{remark}

\section{Proof of Theorem \ref{main theorem}} 
We first derive an integral representation of the solution of  the system (\ref{1.16})-(\ref{1.21}).  The following result 
can be found in many references including \cite{cannon1967cauchy, Fr2}.  For the convenience of the reader, we provide its 
proof. 

\begin{lemma}  Under the assumptions of Theorem \ref{main theorem}, the solution $u$ 
to the system (\ref{1.16})-(\ref{1.21}) associated to  $h$, satisfies  

\begin{equation}\label{uxt}
u(x,t)= \intt N(x,0;t,\t)h(\t) d\t -\intt N(x,s(\t);t,\t) \dot s(\t) d\t+ \int^b_0 N(x,\xi;t,0)u_0(\xi) d\xi,
\end{equation}
where the Neumann function $N$ is defined by
\begin{equation}
\label{N=K+K}
N(x,\xi;t,\tau)=K(x,\xi;t,\tau)+ K(-x,\xi;t,\tau), 
\end{equation}
with
\begin{equation}
K(x,\xi;t,\tau)=\frac{1}{2\sqrt{\pi(t-\tau)}} \exp\bigg(\frac{-(x-\xi)^2}{4(t-\tau)} \bigg), \ \ \tau<t.
\end{equation} 
\end{lemma}
\begin{proof}
By definition, we have 
\begin{equation}\label{N}
N(x,\xi;t,\t)=\frac{1}{2\sqrt{\pi(t-\tau)}}\b[ \exp\bigg (\frac{-(x-\xi)^2}{4(t-\tau)} \bigg) + \exp\bigg(\frac{-(x+\xi)^2}{4(t-\tau)} \bigg)\b], \ \ \t<t.
\end{equation}
Let $ \ep>0$ be a  small constant.
Integrating the Green's identity 
\begin{equation}
\paxi (N \paxi u -u\paxi N) -\partial_\tau (Nu)=0
\end{equation}
over the domain $\{0<\xi<s(\tau), 0<\ep <\tau< t-\ep \}$, we obtain 
\begin{equation}\label{N3}
\int^{t-\ep}_\ep \int_0^{s(\tau)}  \bigg \{\paxi (N \paxi u -u\paxi N) -\partial_\tau (Nu)\b\} d\tau d\xi =0. 
\end{equation}
Hence
\begin{align*}
\int^{t-\ep}_\ep \int_0^{s(\tau)}  \paxi (N \paxi u -u\paxi N)d\tau d\xi &= \int^{t-\ep}_\ep N(x,s(\tau);t,\tau)u_\xi(s(\t),\t)    d\t - \int^{t-\ep}_\ep u(s(\t),\t) N_\xi(x,s(\t); t,\t) d\t \\ 
&-\int^{t-\ep}_\ep N(x,0;t,\t)u_\xi(0,\t) d\t+\int^{t-\ep}_\ep u(0,\tau) N_\xi(x,0;t,\t) d\t. 
\end{align*}
Differentiate (\ref{N}) with respect to $\xi$, we get
\begin{align*}
N_\xi(x,\xi;t,\t)          
=\frac{1}{2\sqrt{\pi(t-\tau)}} \bigg\{&\frac{(x-\xi)}{2(t-\t)}\exp\bigg( \frac{-(x-\xi)^2}{4(t-\tau)} \bigg)
- \frac{(x+\xi)}{2(t-\t)} \exp\bigg(\frac{-(x+\xi)^2}{4(t-\t)}\bigg)  \bigg \}.
\end{align*}
Consequently  $N_\xi(x,0;t,\t)=0$.
Allowing $\ep$ to tend to $0$, we obtain
\begin{align}
\intt \int_0^{s(\t)}  \paxi (N \paxi u -u\paxi N)d\tau d\xi &= \intt N(x,s(\tau);t,\tau)u_\xi(s(\t),\t) d\t - \intt u(s(\t),\t) N_\xi(x,s(\t); t,\t) d\t \nonumber \\ 
&-\intt N(x,0;t,\t)u_\xi(0,\t) d\t.\label{N1}
\end{align}
On the other hand, we have 
\begin{align*}
\int^{t-\ep}_\ep \int_0^{s(\t)} \p_\t (Nu) d\t d\xi &= \int^{t-\ep}_\ep \p_\t \int_0^{s(\t)} Nu d\xi d\t - \int^{t-\ep}_\ep \dot s(\t) N(x,s(\t);t,\t)u(s(\t),\t) d\t \\ &= \int^{s(t-\ep)}_0 N(x,\xi;t,t-\ep)u(\xi,t-\ep) d\xi-\int^{s(\ep)}_0 N(x,\xi;t,\ep)u(\xi,\ep) d\xi  \\ &- \int^{t-\ep}_\ep \dot s(\t) N(x,s(\t);t,\t) u(s(\t),\t) d\t.
\end{align*}
Allowing  again $\ep$ to tend to $0$, we find 
\begin{align}
\intt \int^{s(\t)}_0 \p_\t (Nu) d\t d\xi &= \int^{s(t)}_0 N(x,\xi;t,t)u(\xi,t) d\xi-\int^{s(0)}_0 N(x,\xi;t,0)u(\xi,0) d\xi  \nonumber \\ &- \intt \dot s(\t) N(x,s(\t);t,\t) u(s(\t),\t) d\t.\label{N2}
\end{align}
Substituting (\ref{N1})-(\ref{N2}) into (\ref{N3}), yields
\begin{align*}
 \intt N(x,s(\tau);t,\tau)u_\xi(s(\t),\t) d\t - \intt u(s(\t),\t) N_\xi(x,s(\t); t,\t) d\t -\intt N(x,0;t,\t)u_\xi(0,\t) d\t \\
 -\int^{s(t)}_0 N(x,\xi;t,t)u(\xi,t) d\xi-\int^{s(0)}_0 N(x,\xi;t,0)u(\xi,0) d\xi + \intt \dot s(\t) N(x,s(\t);t,\t) u(s(\t),\t) d\t=0.
\end{align*}
Finally, we obtain 
\begin{equation}
u(x,t)= \intt N(x,0;t,\t)h(\t) d\t -\intt N(x,s(\t);t,\t) \dot s(\t) d\t+ \int^b_0 N(x,\xi;t,0)u_0(\xi) d\xi. 
\end{equation}
\end{proof}

We next reduce the inverse  Stefan problem to solving an ill-posed linear integral equation. \\

Since $u$ is   the unique  solution  to  the system (\ref{1.16})-(\ref{1.21}) associated to  $s$, it 
satisfies  $u(s(t),t)=0$ for all $t \in (0, T)$, which combined to equality \eqref{uxt}, leads to 
\bea
u(s(t),t)= \intt N(s(t),0;t,\t)h(\t) d\t -\intt N(s(t),s(\t);t,\t) \dot s(\t) d\t+ \int^b_0 N(s(t),\xi;t,0)u_0(\xi) d\xi = 0, 
\eea
for all  $ t\in (0, T)$.
Consequently $h(t)$ is a solution to the following principal integral equation
\bean \label{eqq}\\  \nonumber
\intt N(s(t),0;t,\t)h(\t) d\t = \intt N(s(t),s(\t);t,\t) \dot s(\t) d\t- \int^b_0 N(s(t),\xi;t,0)u_0(\xi) d\xi, \textrm{ for all  } t\in (0, T).
\eean
The proof of stability is based on estimating the modulus of continuity of the corresponding linear integral operator.

\begin{lemma}
\label{Lemma_eta}
Under the  assumptions in Theorem \ref{main theorem}, the following inequality 
\begin{equation}\label{2.58}
 \abs{\intt N(x,0;t,\t)[h(\t)- \w h(\t)]d\t } \leq C \norm{s-\tilde{s}}_{W^{1, \infty} }^{1/4}, \textrm{  for all  }x \geq s(t), \;  0 < t < T,
\end{equation}
holds. The constant $C>0$ only depends on  $u_0$, $b$, $T$, $H$,  and  $M$.
\end{lemma}

\begin{proof} 
In the following proof $C$ stands for a generic  constant strictly larger  than zero that only depends 
on  $u_0$, $b$, $T$, $H$,  and  $M$.\\

We deduce from Remark \ref{sbounds}, that there exists $s_\infty>0$ depending only on $b, M$ and $T$, such that 
\bea
\| s\|_{W^{1, \infty}}, \; \| \w s\|_{W^{1, \infty}} \leq  s_\infty.
\eea
Without loss of generality we can assume that $s_\infty \geq 1$.\\

Since $\w u$ is  the unique  solution  to  the system  (\ref{1.16})-(\ref{1.21}),  associated to  $\w s$, $\tilde h,$ it
also satisfies \eqref{eqq} with  $\w h$ and $\w s $ substituting respectively $h$ and $s$  on the right side. 
Taking the difference between the two  linear equations  \eqref{eqq} related  respectively to  $u$ and $\w u$,  gives 

\begin{align*}
&  \intt N(s(t),0;t,\t)[h(\t)- \w h(\t)]d\t =   \intt [N(s(t),0;t,\t)- N(\w s(t),0;t,\t)]\w h(\t) d\t \\
& - \intt [N(\w s(t),\w s(\t);t,\t)- N(s(t),s(\t);t,\t)] \dot {\w s}(\t) d\t - \intt N(s(t),s(\t);t,\t)[\dot {\w s}(\t)-\dot s(\t)] d\t \\ 
&- \int^b_0 [N(s(t),\xi;t,0)-N(\w s(t),\xi;t,0)] u_0(\xi) d\xi  =  \sum_{i=1}^4 I_i,
\end{align*}
for all  $ t\in (0, T)$.\\

Now, we shall estimate each of the integrals  $I_i, i=1,\cdots, 4,$ in terms of the difference between 
$s$ and $\w s$.\\ 

Let $\ep$ be a fixed  constant  satisfying $0<\ep<t < T$, and set 
\begin{align*}
I_1^\ep = \int_{0}^{t-\ep } [N(s(t),0;t,\t)- N(\w s(t),0;t,\t)]\w h(\t) d\t. 
\end{align*} 

From the mean value Theorem, we deduce 
\begin{equation}\label{exponentiel}
\abs{e^{-a }-e^{-b} } \leq e^{-\min(a,b) } \abs{a-b}, \quad  a, b\geq 0.
\end{equation}
Recalling that $s(t), \w s(t) > b$ for all $t \in [0, T]$, and applying  the inequality \eqref{exponentiel} to 
the integrand of $I_1^\ep$, we obtain 

\begin{align}
\abs{I_{1}^\ep} &= \abs{ \int_{0}^{t-\ep } \frac{1}{2\sqrt{\pi(t-\tau)}}\b[ 2\exp\bigg (\frac{-s(t)^2}{4(t-\tau)} \bigg)  -2\exp\bigg (\frac{-\w s(t)^2}{4(t-\tau)} \bigg)\b] \w h(\t)  d\tau} \no \\
&\leq \int_{0}^{t-\ep } \frac{e^{\frac{-b^2}{4(t-\tau)}}}{\sqrt{\pi(t-\tau)}} \abs{\dfrac{s^2(t)- \tilde{s}^2(t)}{4(t-\tau)}} \abs{\tilde{h}(\tau)} d\tau \no \\
&\leq \int_{0}^{t-\ep} \frac{e^{\frac{-b^2}{4(t-\tau)}}}{2\sqrt{\pi} (t-\tau)^{3/2} }d\tau  \max{ \big( \norm{s}_{\infty},\norm{\tilde{s}}_{\infty} \big) } \norm{\tilde{h}}_{\infty} \norm{s-\tilde{s}}_{\infty}  \no \\
&\leq \int_{0}^{T} \frac{e^{\frac{-b^2}{4r}}}{2\sqrt{\pi} (r)^{3/2} }dr  \max{ \big( \norm{s}_{\infty},\norm{\tilde{s}}_{\infty} \big) } \norm{\tilde{h}}_{\infty} \norm{s-\tilde{s}}_{\infty}  \no 
\end{align}
Hence
\bean \label{III1}
\abs{I_{1}} \leq C \norm{s-\tilde{s}}_{\infty}\label{I21}.
\eean
Now, let  
\begin{align}
I_2^\ep  = \int_{0}^{t-\ep } [N(\w s(t),\w s(\t);t,\t)- N(s(t),s(\t);t,\t)] \dot {\w s}(\t) d\t, \no
\end{align}
or equivalently 
\begin{align}
I_{2}^\ep
&= \int_{0}^{t-\ep } \frac{1}{2\sqrt{\pi(t-\tau)}}\b[ \exp\bigg (\frac{-(\w s(t)-\w s(\t))^2}{4(t-\tau)} \bigg)\nonumber 
-\exp \b( \frac{-(s(t)- s(\t))^2}{4(t-\tau)} \b) \\
&+ \exp\bigg(\frac{-(\w s(t)+\w s(\t))^2}{4(t-\tau)}\b) 
 - \exp\bigg(\frac{-( s(t)+ s(\t))^2}{4(t-\tau)} \b) \b]  \dot {\w s}(\t) d\t.  \nonumber
\end{align}
For $0 <\tau < t-\ep $, using (\ref{exponentiel}), gives 
\begin{align}
&\abs{\exp\bigg (\frac{-(\w s(t)-\w s(\t))^2}{4(t-\tau)} \bigg)- \exp\bigg(\frac{-( s(t)- s(\t))^2}{4(t-\tau)}\b)}
 \leq \frac{1}{4(t-\t)} \abs{(\w s(t)-\w s(\t))^2 -( s(t)- s(\t))^2} \no \\
& \leq  \frac{1}{4(t-\t)} \abs{(\w s(t)-\w s(\t))-(s(t)-s(\t))} \abs{(\w s(t)-\w s(\t))+(s(t)-s(\t))}.  \no 
\end{align}

Since  $s, \w s$ are Lipschitz functions, we have 
\bean\label{lipschitz}
\frac{\abs{s(t)- s(\tau)}}{t-\tau}, \; \frac{\abs{\w s(t)-\w s(\tau )}}{t-\tau} 
\leq \max(\norm{ s}_{C^{0,1}([0,T])}, \norm{ \w s}_{C^{0,1}([0,T])}) \leq C  \max(\norm{s}_{W^{1, \infty},} \norm{\w s}_{W^{1, \infty}}) .
\eean
Therefore
\begin{align*}
&\abs{\exp\bigg (\frac{-(\w s(t)-\w s(\t))^2}{4(t-\tau)} \bigg)- \exp\bigg(\frac{-( s(t)- s(\t))^2}{4(t-\tau)}\b)}
 \leq  C\max(\norm{s}_{W^{1, \infty},} \norm{\w s}_{W^{1, \infty}}) \norm{s-\w s}_\infty.
\end{align*}
Consequently 
\begin{align*}
\abs{I_{2}^\ep} \leq C \int_{0}^{t-\ep } \frac{1}{2\sqrt{\pi(t-\tau)}} d\t \max(\norm{s}_{W^{1, \infty},} \norm{\w s}_{W^{1, \infty}}) 
\norm{ \w s}_{W^{1, \infty}}\norm{s-\w s}_\infty  \leq C \parallel s-\w s \parallel_{\infty}, 
\end{align*}
which leads to 
\begin{align}
\abs{I_{2}} \leq C \parallel s-\w s \parallel_{\infty}, \label{I31}
\end{align}

On the other hand, we have 
\begin{align}
|I_3|&= \b|  \intt N(s(t),s(\t);t,\t)[\dot {\w s}(\t)-\dot s(\t)] d\t\b| \no\\
&=\b|\intt \frac{1}{2\sqrt{\pi(t-\tau)}}\b[ \exp\bigg (\frac{-(s(t)-s(\t))^2}{4(t-\tau)} \bigg) + 
\exp\bigg(\frac{-(s(t)+s(\t))^2}{4(t-\tau)} \bigg)\b]\b(\dot{\w s}(\t)-\dot s(\t) \b) d\t \b| \no \\
&\leq  \intt \frac{1}{\sqrt{\pi(t-\tau)}} \abs{\dot{\w s}(\t)- \dot s(\t)} d\tau  
\leq \intt \frac{1}{\sqrt{\pi(t-\tau)}} d\tau \parallel \dot{\w s}- \dot s\parallel_{\infty} \leq \frac{2}{\sqrt{\pi}} \sqrt{T} \parallel \dot{\w s}- \dot s\parallel_{\infty} \no \\
&\leq C \parallel s- \w s\parallel_{W^{1, \infty}}. \label{I4}
\end{align}
Next, we estimate $I_4$. 
\begin{align*}
|I_4|&= \b |\int^b_0 [N(s(t),\xi;t,0)-N(\w s(t),\xi;t,0)] u_0(\xi) d\xi \b| \no \\
&=\b | \int^b_0 \frac{1}{2\sqrt{\pi t}}\b[ \exp\bigg (\frac{-(s(t)-\xi)^2}{4t} \bigg) + \exp\bigg(\frac{-(s(t)+\xi)^2}{4t} \b)\\
&- \exp\bigg (\frac{-(\w s(t)-\xi)^2}{4t} \bigg) \no 
-\exp\bigg(\frac{-(\w s(t)+\xi)^2}{4t} \b)\b]u_0(\xi) d\xi \b| \no \\
&\leq \int^b_0 \frac{1}{2\sqrt{\pi t}} \abs{\exp\bigg (\frac{-(s(t)-\xi)^2}{4t} \bigg)-\exp\bigg (\frac{-(\w s(t)-\xi)^2}{4t} \bigg)} \abs{u_0(\xi)} d\xi \\
&+\int^b_0 \frac{1}{2\sqrt{\pi t}} \abs{\exp\bigg(\frac{-(s(t)+\xi)^2}{4t} \b)- \exp\bigg(\frac{-(\w s(t)+\xi)^2}{4t} \bigg) }  \abs{u_0(\xi)} d\xi.
\end{align*}
Since $s(t)>0$, applying inequality  \eqref{exponentiel} to the estimate of $I_4$, yields 
\bea 
|I_4| &\leq& \frac{1}{8\sqrt{\pi}} \int^b_0  \left( \abs{(s(t)+\xi)^2 - (\w s(t)+\xi)^2 } +\abs{(s(t)-\xi)^2 - (\w s(t)-\xi)^2 }
 \right)\frac{e^{-\frac{\xi^2}{t}}}{t^{3/2}} \abs{u_0(\xi)} d\xi \\
&\leq& \frac{1}{8\sqrt{\pi}} \int^b_0 \abs{s(t)-\w s(t)} \left( \abs{s(t)+\w s(t)-2\xi}+ \abs{s(t)+\w s(t)-2\xi}
\right)\frac{e^{-\frac{\xi^2}{t}}}{t^{3/2}} d\xi\Vert u_0\Vert_{\infty}\\
&\leq& \frac{1}{2\sqrt{\pi}} \int^b_0 \frac{e^{-\frac{\xi^2}{t}}}{ \sqrt t } d\xi
\frac{\abs{s(t)-\w s(t)}}{t} \max (\norm{s}_{\infty},\norm{\w s}_{\infty})  \Vert u_0\Vert_{\infty}\\
&\leq& \frac{1}{4} \frac{\abs{s(t)-\w s(t)}}{t} \max (\norm{s}_{\infty},\norm{\w s}_{\infty})  \Vert u_0\Vert_{\infty}.
\eea
Since $s(b)= \w s(b)$, we have 
\bea
\frac{\abs{s(t)-\w s(t)}}{t}\leq 2 \norm{ s- \w s}_{C^{0,1}([0,T])} \leq C   \norm{ s- \w s}_{W^{1, \infty} }.
\eea
Therefore
\bean  \label{I5}
|I_4| \leq C \norm{ s- \w s}_{W^{1, \infty} }.
\eean

By combining \eqref{III1}, (\ref{I31}), (\ref{I4}) and (\ref{I5}), we finally obtain 

\begin{align}
|\sum_{i=1}^4 I_i | & \leq C \norm{s-\w s}_{W^{1, \infty}}.
\end{align}

By  the maximum principle  for the heat equation \cite{{cannon1967cauchy} }, we have
\begin{equation}\label{2.17}
\abs{\intt N(x,0;t,\t)[h(\t)- \w h(\t)]d\t } \leq C \norm{s-\w s}_{W^{1, \infty}},\quad x \geq s(t),   \quad  0 \leq t \leq T,
\end{equation}
which ends the proof of the lemma. 
\end{proof}


\begin{lemma}
\label{sss}
Under the  assumptions in Theorem \ref{main theorem}, there exists a constant
$C>0$  that only depends on   $u_0$, $b$, $T$, $H$,  and  $M$,
such that the following inequality 
\begin{equation}\label{eq000}
 \abs{\intt N(0,0;t,\t)[h(\t)- \w h(\t)]d\t } \leq \frac{C}{\ln\left(  \abs{\ln \norm{s-\w s}_{W^{1, \infty}}}\right)}.
\end{equation}
holds for all $t\in (0, T)$.
\end{lemma}
\begin{proof}
 Further $C>0$  denotes  a generic  constant  that only depends 
on  $u_0$, $b$, $T$, $H$,  and  $M$.\\

Remark \ref{sbounds} implies 
\bea
s(t), \, \w s(t) \leq s_\infty, \textrm{ for all  } t\in (0, T).
\eea

Let $\Omega$ be the triangle defined by $\Omega= \{ z \in \C,\,  0<\re(z) < 2s_\infty, \; |\im(z)| < |\re(z)| \}$. \\

For $z \in \Omega$, define  
\bean \label{Fz}
F(z,t)=\ds\intt N(z,0;t,\t)(h(\t) -\w h(\t))d\t  = \frac{1}{\sqrt{\pi}} \int^{t}_0 \frac{1}{\sqrt{t-\t}} 
\exp\bigg (\frac{-z^2}{4(t-\t)} \bigg)(h(\t) -\w h(\t)) d\t,
\eean
We remark that the function $ z \rightarrow F(z, t)$ is holomorphic in $ \Omega$, and satisfy
\bean \label{finfty}
|F(z, t)| \leq  F_\infty, \quad  z\in \Omega,
\eean 
where $ F_\infty = \frac{4M\sqrt{T}}{\sqrt \pi }.$\\

 Denote $w_0(z)$,    the harmonic measure of $ \Omega \setminus [s_\infty, 2s_\infty] \times\{0\}$. It is the unique 
 solution  to the system:
\bean \label{harmonic1}
\Delta w_0(z)&=&0 \quad z\in \Omega \setminus [s_\infty, 2s_\infty) \times\{0\},\\\label{harmonic2}
w_0(z) &=& 0 \quad z\in [s_\infty, 2s_\infty) \times\{0\},\\\label{harmonic3}
w_0(z) &=& 1 \quad z\in \partial \Omega.
\eean

  The 
 holomorphic unique continuation of the functions $ z \rightarrow F(z, t)$ using the Two constants Theorem \cite{Ne, BT}, gives
\[
|F(z, t)| \leq \left( F_\infty \right)^{1-w_0(z)} \left( \sup_{r \in [s_\infty, 2s_\infty] \times\{0\} }  |F_\ep(r, t)| \right)^{w_0(z)},
\; \; z\in \Omega \setminus [s_\infty, 2s_\infty] \times\{0\}.
\]

Combining the previous inequality with the estimate in  Lemma \ref{Lemma_eta}, yields 
\bean \label{fff}
|F(z, t)| \leq C 
\norm{s-\w s}_{W^{1, \infty}}^{\frac{w_0(z)}{4}},  \; \; z\in \Omega \setminus [s_\infty, 2s_\infty] \times\{0\}.
\eean
\begin{lemma} Let $w_0(z)$ be the unique solution to the system \eqref{harmonic1}, \eqref{harmonic2}, 
\eqref{harmonic3}. Then, there 
exist  constants  $c, C>0$ that  only depend  on $s_\infty$,
such that the following inequality 

\bean \label{lowerboundw}
w_0(x) \geq \frac{C}{x}e^{-\frac{c}{x}}, 
\eean
holds for all $x\in (0, s_\infty)$.

\end{lemma}
\begin{proof}
For $a\in (0, s_\infty)$, let  $\mathcal R_a  \subset \Omega $  be the rectangle defined by  $ \mathcal R_a = 
 \{ z \in \C,\,  \frac{a}{2}<\re(z) < 2s_\infty, \; |\im(z)| \leq \frac{a}{2} \},$ and denote  
$\w w_0(z, a) $  the harmonic measure of $ \mathcal  R_a   \setminus [s_\infty, 2s_\infty] \times \{0\}$. By Maximum 
principle it follows that  

\bean \label{ineqqq}
w_0(z) \geq  \w w_0(z, a), 
\eean
for all $z\in \mathcal R_a$. We deduce from Lemma 3.1 in \cite{AA} the following estimate
\bean \label{hhj}
\w w_0(a, a) \geq  \frac{C}{a}e^{-\frac{c}{a}}, 
\eean
holds for all $a\in (0, s_\infty)$, where $c, C>0$ only depend on $s_\infty$.\\

Taking $z=a$ at inequality \eqref{ineqqq}, and using \eqref{hhj}, we obtain  the wanted estimate.

\end{proof}

Applying the estimate \eqref{lowerboundw} to inequality \eqref{fff}, yields 
\bean \label{fff2}
|F(x, t)| \leq C 
\norm{s-\w s}_{W^{1, \infty}}^{  \frac{C}{x}e^{-\frac{c}{x}}},  \; \; x\in (0, s_\infty).
\eean

On the other hand, we have
\bean \label{ggg}
| F(x, t) - F(0, t) | &\leq  \frac{1}{\sqrt{\pi}}  \int_{0}^t \frac{1}{\sqrt{t-\t}}\left(1- \exp (\frac{-x^2}{4(t-\t) }) \right) d\t \leq 
   \frac{4M}{\sqrt{\pi}} \int_{0}^T\frac{1}{(t-\t)^{\frac{3}{4}}}  d\t  \sqrt x  \leq C \sqrt x, \label{eee}
\eean
for all  $t\in (0, T)$. \\

We deduce from inequalities \eqref{fff2} and \eqref{ggg}, the following estimate  
\bean \label{rrrr}
|F(0, t)| \leq C \left(
\norm{s-\w s}_{W^{1, \infty}}^{\frac{C}{x}e^{-\frac{c}{x}}}+ \sqrt {x}\right),
\eean
for all $x$ in $(0, s_\infty)$ and $t\in (0, T)$. \\

 Finally,  minimizing for fixed $t\in (0, T)$, the right hand side in \eqref{rrrr}  with respect to $x  \in (0, s_\infty)$, we  obtain
 \bea
|F(0, t)| \leq \frac{C}{\ln\left(  \abs{\ln \norm{s-\w s}_{W^{1, \infty}}}\right)},
 \eea 
 for all  $t\in (0, T)$,  which achieves the proof of the Lemma.

\end{proof}

We are now ready  to prove Theorem (\ref{main theorem}).
\begin{proof}
Let 
\bea
\mathfrak H (t) = h(t) - \w h(t). 
\eea
Since $F(0, t) \in L^\infty((0, T))$, according to Theorem 8.1.1 in \cite{cannon1984one}, \cite{Br},  the Abel integral  
equation of first kind
\begin{equation}\label{Fot}
F(0,t)=\frac{1}{\sqrt{\pi}} \intt\frac{\mathfrak H(\t)}{\sqrt{t-\t}} d\t, 
\end{equation}
admits a unique solution,  $\mathfrak H(t) \in  L^1((0, T))$, given by
\begin{align} \label{abelinversion}
\mathfrak H(t) =  \frac{1}{\sqrt{\pi}} \frac{d}{dt} \intt \frac{F(0,\t)}{\sqrt{t-\t}} d\t, \textrm{ for all } t\in (0, T).
\end{align}
In addition, we deduce from \eqref{Fot}, \eqref{eq000}, and \eqref{equiv}
\bea
 \|\mathfrak  H \|_{-\frac{1}{2}} \leq C \|F(0, \cdot)\|_{L^2} \leq \frac{C}{\ln\left(  \abs{\ln \norm{s-\w s}_{W^{1, \infty}}}\right)},
\eea
which achieves the proof of the main Theorem.
\end{proof}

\section{Numerical Analysis}
\subsection{The direct problem}
\subsubsection{Boundary immobilisation method} 
This method consists on fixing the moving interface and  mapping  the moving interval into a fixed one by using the spatial coordinate change \cite{Gol}
 \begin{equation}\label{13n}
  \xi=\frac{x}{s(t)}.
 \end{equation}  
 This approach  was applied to  many moving boundary problems \cite{FFL}.
Now, we reformulate the problem defined by (\ref{1.16}) and (\ref{1.21}) using the new coordinate system $(\xi,t)$ in the fixed 
domain $[0,1]\times [0,T]$. We get

 \begin{eqnarray}
\left\{ \begin{array}{rll}
  \displaystyle \frac{\partial u}{\partial t} &= \displaystyle \frac{\xi}{s} \frac{ds}{dt}\frac{\partial u}{\partial \xi}+\frac{1}{s^2} \displaystyle \frac{\partial^2 u}{\partial \xi^2},\quad &0<\xi<1,\quad t>0 \label{14n} \\
\displaystyle \frac{\partial u}{\partial \xi}(0,t) &=-s(t)  h(t),\quad &t>0, \\
 u(1,t)&=0, \quad &t>0 \\
 \displaystyle \frac{ds}{dt} &=-\displaystyle \frac{1}{s} \displaystyle \frac{\partial u}{\partial \xi}, \quad &t>0 \\
 u(\xi,0)&= u_0(b\xi), \quad &0<\xi<1 \\
  s(0)&=b. \quad  &t=0 
\end{array}\right.
\end{eqnarray}
We subdivide the spatial interval $[0,1]$ into $N$ equal intervals, we set $h=\Delta \xi=\frac{1}{N}$ and $\xi=ih$. We also subdivide the time interval $[0,T]$ into $M$ equal intervals, and we set $t_m=mk$,  $k=\Delta t =\displaystyle \frac{1}{M}$ and $r=\displaystyle \frac{k}{h^2}$. 
Next, we  use an explicit Euler method in time  and a finite differences method in space, to obtain 
\begin{equation}\label{18n}
 U_i^{m+1}=U_i^m+ \displaystyle \frac{rh\xi_i\dot s_m}{2s_m}(U_{i+1}^m-U_{i-1}^m) + \displaystyle \frac{r}{s_m^2} (U_{i+1}^m-2U_i^m + U_{i-1}^m)
\end{equation} 

 Applying a central differences at the boundary conditions (\ref{1.17}), we get
\begin{equation}\label{19m}
 U_{i+1}^m=U_{i-1}^m+2h \cdot h(t_m).
\end{equation}
 Taking $i=0$ in the equations (\ref{18n}) and (\ref{19m}), we obtain
\begin{equation}
\left\lbrace 
\begin{array}{l}
 U_{-1}^m=U_{1}^m+2h \cdot h(t_m) \\ \\
 U_0^{m+1}=U_0^m+\displaystyle \frac{rh\xi_0 \dot s_m}{2s_m}(U_{1}^m- U_{-1}^m) +\displaystyle \frac{r}{s_m^2}(U_{1}^m-2U_0^m+ U_{-1}^m) 
 \end{array}
\right.
\end{equation}
  Eliminating the term $U_{-1}^m$ in the  equations above, leads to
$$U_0^{m+1}=(1-\displaystyle \frac{2r}{s_m^2})U_0^m+ \displaystyle \frac{2r}{s_m^2} U_1^m+(\displaystyle \frac{2hr}{s_m} -k\xi_0\dot s_m)h(t_m).$$
Finally, applying a change of variable $Z(t)=s^2(t)$, yields  
\begin{equation}
\left\lbrace 
\begin{array}{l}
 U_1^{m+1}=(1-\displaystyle \frac{2r}{Z_m})U_1^m+\displaystyle \frac{2r}{Z_m}U(2,m)+\displaystyle \frac{4hr-k\xi_1 \dot Z_m}{2\sqrt {Z_m}}h(t_m); \ \ m=1,...,M \\ \\ 
    U_{N+1}^{m+1}=0; \ \ m=1,...,M \\ \\ 
 U_i^{m+1}=U_i^m+ \displaystyle \frac{rh\xi_i \dot Z_m}{4Z_m} (U_{i+1}^m-U_{i-1}^m)+
\displaystyle \frac{r}{Z_m}(U_{i+1}^m-2U_i^m + U_{i-1}^m); \ \ i= 2,...,N \\ \\
    Z_{m+1}=Z_m-\displaystyle \frac{k}{h}(3U_{N+1}^m-4U_{N}^m+U_{N-1}^m);  \ \ m=1,...,M  \\ \\ 
    \dot Z_{m+1}=\displaystyle \frac{Z_{m+1}-Z_m}{dt} \ \ m=1,...,M
 \end{array}
\right.
\end{equation}  
For numerical stability requirements, we apply  Von-Neumann approach \cite{smith1985numerical} to obtain  bounds on the size of the time step $k$
\begin{equation*}
k\leq h^2 \w Z, \quad \text{ where } \w Z=\underset{0<t<T}{\inf} Z(t).
\end{equation*}

\subsubsection{Numerical results} 

We consider the Stefan problem (\ref{1.16})-(\ref{1.21}) with initial and boundary conditions for which we know the exact solution in order to measure the performance of the numerical method mentioned above. Therefore, we take 
\begin{equation}
[0,T]=[0,1], \quad h(t)=exp(t), \quad u_0(x)= exp(-x)-1 , \quad b=0.1 
\end{equation}
The associated exact solution is 
\begin{equation}
\left\lbrace 
\begin{array}{l}
u(x,t)=exp(t-x)-1\\
s(t)=t.
 \end{array}
\right.
\end{equation}
The tables below provide the results obtained by the Boundary Immobilisation Method with a comparison
 between the exact solution and the approximate one at a fixed time $ t_m $.

 \textbf{Table 1} \\
Values of the temperature distribution as predicted by the numerical and exact solutions at a fixed time,
where the error is defined by
\begin{equation}
\| e\|_u= \frac{1}{N} \sum_{i=0}^{N-1} \bigg|1-\frac{U_i^m}{u(x_i,t_m) }\bigg|.
\end{equation}

  \begin{center}
\begin{tabular}{lllllll}
\hline 
$\xi$ & Numerical &solution & \ \ & \ \ & exact & \ \   \\ 
\hline 
\ \ & N=10 & N=20 & N=40 & N=80 & solution & Error \\
\hline 
0.1 & 0.094193 & 0.094187 & 0.094177 & 0.094175 & 0.094175 & 0.11326$\times 10^{-3}$\\ 
0.2 & 0.083340 & 0.083299 & 0.083289 & 0.083287 & 0.083288 & 0.26014$\times 10^{-3}$\\
0.3 & 0.072554 & 0.072518 & 0.072510 & 0.072508 & 0.072509 & 0.25743$\times 10^{-3}$\\
0.4 & 0.061875 & 0.061845 & 0.061838 & 0.061837 & 0.061837 & 0.25335$\times 10^{-3}$\\
0.5 & 0.051303 & 0.051278 & 0.051272 & 0.051271 & 0.051271 & 0.26013$\times 10^{-3}$\\
0.6 & 0.040836 & 0.040816 & 0.040812 & 0.040811 & 0.040811 & 0.25318 $\times 10^{-3}$\\
0.7 & 0.030473 & 0.030458 & 0.030455 & 0.030454 & 0.030455 & 0.24079 $\times 10^{-3}$\\
0.8 & 0.020213 & 0.020204 & 0.020202 & 0.020201 & 0.020201 & 0.26401 $\times 10^{-3}$\\
0.9 & 0.010056 & 0.010051 & 0.010050 & 0.010050 & 0.010050 & 0.23217$\times 10^{-3}$\\
1 & 0 & 0 & 0 & 0 & 0 & 0\\ 
\hline 
  \end{tabular}
  \end{center}

\textbf{Table 2} \\
 Values of the location of the moving interface as predicted by the numerical  and exact solutions at a fixed time,
where the relative errors of  $s$ and $\dot s$ are respectively defined by 
 \begin{align*}
 e_{s}^m &=\bigg| 1-\frac{S_m}{s(t_m)} \bigg|, \\
   e_{\dot s}^m &=\bigg| 1-\frac{\dot S_m}{\dot s(t_m)} \bigg|. 
\end{align*}  

\begin{center}
\begin{tabular}{lllll}
\hline 
N & $S_m$ & Error & $\dot S_m$ & Error\\ 
\hline 
10 & 0.10000080 & 0.61673$\times 10^{-3}$& 0.99934974 & 0.65025$\times 10^{-3}$ \\
20 & 0.10001384 & 0.13847 $\times 10^{-3}$ & 0.99985316 & 0.14683$\times 10^{-3}$ \\
40 & 0.10000328 & 0.32871$\times 10^{-4}$ & 0.99996504 & 0.34958$\times 10^{-4}$\\
80 & 0.10000080 & 0.80114$\times 10^{-5}$ & 0.99999146 & 0.85327 $\times 10^{-5}$ \\
\hline 
 Exact solution  & 0.1 & \ \  \ \ & 1 \\
\hline 
\end{tabular}
\end{center}

It is clearly observed that all numerical predictions are in good agreement with the exact solution and that our numerical schemes are convergent.

\subsection{The inverse problem}
\subsubsection{Linear integral equations}
Since $u(s(t),t)=0$, using the equation (\ref{uxt}), we have 
\begin{equation}\label{ust}
 \intt N(s(t),0;t,\t)h(\t) d\t =\intt N(s(t),s(\t);t,\t) \dot s(\t) d\t - \int^b_0 N(s(t),\xi;t,0)u_0(\xi) d\xi,
\end{equation} 
We consider a uniform grid of the temporal interval $[0,T]$, with a time step $\Delta t= \frac{T}{N}$, that is,  $t_j=j\Delta t, \  j=0,...,N$. We also
use a uniform grid of spatial interval $[0,b]$ with a spatial step $\Delta \xi =\frac{b}{M}$, that is,  $\xi_i= i \Delta \xi, \  i=0,...,M$. Then the equation (\ref{ust}) becomes
\begin{align*}
\sum_{j=0}^{N-1}\int_{t_j}^{t_{j+1}} N(s(t_i),0;t_i,\t)h(\t) d\t = \sum_{j=0}^{N-1} \int_{t_j}^{t_{j+1}} N(s(t_i),s(\t);t_i,\t) \dot s(\t) d\t - \sum_{k=0}^{M-1}\int^{\xi_{k+1}}_{\xi_k} N(s(t_i),\xi;t_i,0)u_0(\xi) d\xi. 
\end{align*} 
For example if we use a quadrature formula on one point, we get for $i=1,...,N$ 
\begin{align}\label{sum_of_N}
\sum_{j=0}^{N-1} N(s(t_i),0;t_i,\t_j)h(\t_j) \Delta t = \sum_{j=0}^{N-1}  N(s(t_i),s(\t_j);t_i,\t_j) \dot s(\t_j) \Delta t - \sum_{k=0}^{M-1} N(s(t_i),\xi_k;t_i,0)u_0(\xi_k) \Delta \xi, 
\end{align} 
where $\t_j \in [t_j,t_{j+1}]$ for $j=1,...,N-1$ and $\xi_k \in [ \xi_k,\xi_{k+1}]$ for $k=1,...,M-1$. \\

This system can be represented by  the following algebraic linear system
\begin{equation}\label{Ahg}
\mathcal A h=g, 
\end{equation}
where $\mathcal A$ denotes a matrix depending on the quadrature formulas,
$h:=(h_j)=(h(\t_j))$ denotes the vector of unknown Neumann condition of the inverse problem 
(\ref{Inverse problem}), and $g:=(g_i)$ is the vector representing the right hand side of the equation 
(\ref{sum_of_N}).\\

Since the obtained system corresponds to an ill-posed inverse problem, the matrix $\mathcal A$ is
 ill-conditioned, and the system need to be regularized.
We consider here the Tikhonov regularization method, which solves the modified system of equation 
\begin{equation}\label{Tikhonov}
(\mathcal A^{tr} \mathcal A+\lambda I)h= \mathcal A^{tr} g+\lambda h_{exact},
\end{equation}
where the superscript $^{tr}$ denotes the transpose of a matrix, $I$ the identity matrix, and $\lambda > 0$
 is the small regularization parameter.

\subsubsection{Numerical results}
\paragraph{ {\bf Example 1}}
The first example has a moving boundary given by the linear function
\begin{equation}
s(t)=\sqrt{2}-1+\frac{t}{\sqrt{2}}, \ t\in[0,T].
\end{equation}
We take the exact solution given by 
\begin{equation}
u(x,t)=-1+\exp \big( 1-\frac{1}{\sqrt{2}} +\frac{t}{2} -\frac{x}{\sqrt{2}} \big) , \ \ [x,t]\in [0,s(t)]\times [0,1].
\end{equation}
Therefore, this example has the following initial and boundary conditions
\begin{align}
b&=s(0)=\sqrt{2}-1, \\
u_0(x)&= -1+\exp \big( 1-\frac{1}{\sqrt{2}} -\frac{x}{\sqrt{2}} \big), x\in[0,b], \\ 
u(s(t),t)&=0, \  t\in(0,1], \\
u_x(s(t),t)&= \dot s(t)=-\frac{1}{\sqrt{2}}, \ t\in(0,1].
\end{align}
In this example we wish to recover the Neumann boundary condition along the fixed boundary $x = 0$ given by 
\begin{equation}
h(t)=\frac{1}{\sqrt{2}} \exp\big ( 1-\frac{1}{\sqrt{2}} +\frac{t}{2} \big), \ t\in[0,1].
\end{equation}

This example has been previously tested in \cite{johansson2011method} using the method of fundamental solutions.

In a first time, we use Gauss-Legendre formulas in order to 
\begin{figure}[H]
\centering
\includegraphics[width=12.5cm,height=7.5cm]{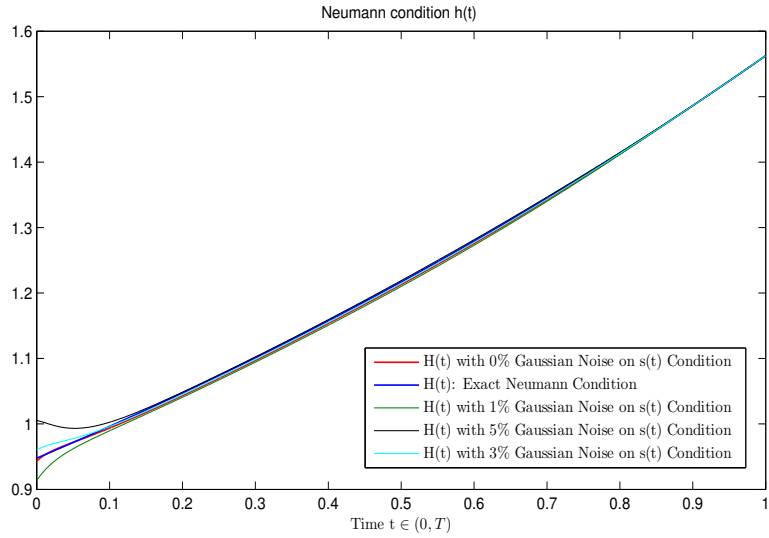} 
\caption{The exact solution $-u_x(0, t)$ and approximate solution with different Gaussian noise levels obtained 
with $\lambda=10^{-3}$ and $N=1000$ using Gauss-Legendre Formula.}
\label{HGauss}
\end{figure}

\begin{table}[H]
\begin{tabular}{|c|c|c|}
\hline
 $\lambda$ & Noise on $s(t)$ ($\%$) &$\dfrac{\ds\norm{h_{exact} - h}_2}{\ds\norm{h_{exact}}_2}$ \\
\hline
 0.001 & 0 $\%$ & 0.0025 \\
\hline
 0.001 & 1 $\%$ & 0.0056 \\
 \hline
 0.001 & 3 $\%$ & 0.0078\\
\hline
 0.001 & 5 $\%$ & 0.0124 \\
\hline
\end{tabular}
\caption{Relative errors for all cases.}
\label{tableGauss}
\end{table}

\begin{figure}[H]
\centering
\includegraphics[width=12.5cm,height=7.5cm]{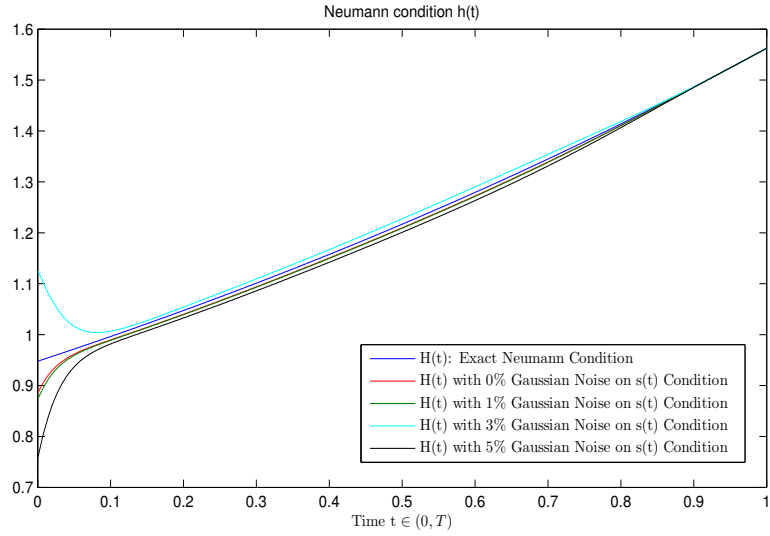}
\caption{The exact solution $-u_x(0, t)$ and approximate solution with different Gaussian noise levels obtained with $\lambda=10^{-3}$ and $N=1000$ using mid-point formula.}
\label{Hpoint_milieu}
\end{figure}

\begin{table}[H]
\begin{tabular}{|c|c|c|}
\hline
 $\lambda$ & Noise on $s(t)$ ($\%$) &$\dfrac{\ds\norm{h_{exact} - h}_2}{\ds\norm{h_{exact}}_2}$ \\
\hline
 0.001 & 0 $\%$ & 0.0073 \\
\hline
 0.001 & 1 $\%$ & 0.0086 \\
 \hline
 0.001 & 3 $\%$ & 0.0209 \\
\hline
 0.001 & 5 $\%$ & 0.0211 \\
\hline
\end{tabular}
\caption{Relative errors for all cases.}
\label{table_point_milieu}
\end{table}

\begin{figure}[H]
\centering
\includegraphics[width=12.5cm,height=7.5cm]{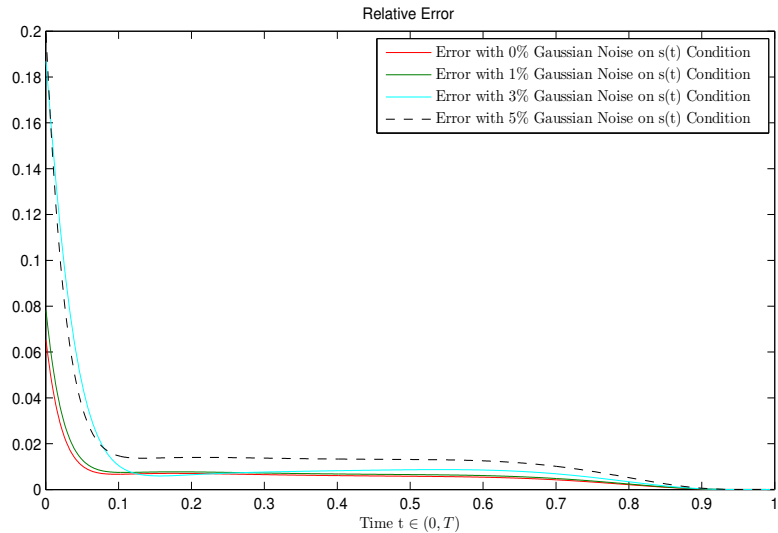}
\caption{The relative error with different Gaussian noise levels obtained with $\lambda=10^{-3}$, $N=1000$.}
\label{error_point_milieu}
\end{figure}

In Figures (\ref{HGauss}) and (\ref{Hpoint_milieu}), we present the exact and approximate solution using the linear equation (\ref{Tikhonov}). 
In order to test the stability of our inverse problem, we add a different level of gaussian noise to the data $s(t)$.  The Tables (\ref{table_point_milieu})-(\ref{tableGauss}) show that the accuracy is altered with noise and the relative error on the numerical solution 
is under $0.03$.  \\

The Figures (\ref{Hpoint_milieu})-(\ref{error_point_milieu}) and the table (\ref{table_point_milieu}) are obtained by using mid-point rule.  
In  Figure (\ref{HGauss}), we refined the error on the integration and use Gauss-Legendre formulas of the third order. 
We remark that we have a better stability around $t=0$. \\

In  Figure (\ref{error_point_milieu}), we present the evolution  of the relative errors with different gaussian noise levels.  
We remark that the numerical solutions are stable  far away from $0$, and they become more accurate as the amount of noise decreases.
The numerical results confirms the theoretical predictions of the Theorem (\ref{main theorem}) which relates to the stability of the Neumann condition $h$ with respect  to the free boundary $s$.


\paragraph{{\bf Example 2}}
In this example the moving boundary is given by
\begin{equation}
s(t)=t+b, \ t\in[0,1].
\end{equation}
We take the exact solution given by 
\begin{equation}
u(x,t)=\exp \big( t-x+b \big) -1 , \ \ [x,t]\in [0,s(t)]\times [0,1].
\end{equation}
Therefore, this example has the following initial and boundary conditions
\begin{align}
b&=s(0), \\
u_0(x)&= \exp \big( b-x \big) -1 , x\in[0,b], \\ 
u(s(t),t)&=0, \  t\in(0,1], \\
u_x(s(t),t)&= \dot s(t)=1, \ t\in(0,1].
\end{align}
We aim to recover the Neumann boundary condition at the fixed boundary $x = 0$ given by 
\begin{equation}
h(t)= \exp\big ( t+b \big), \ t\in[0,1].
\end{equation}

\begin{figure}[H]
\centering
\includegraphics[width=11.5cm,height=7.5cm]{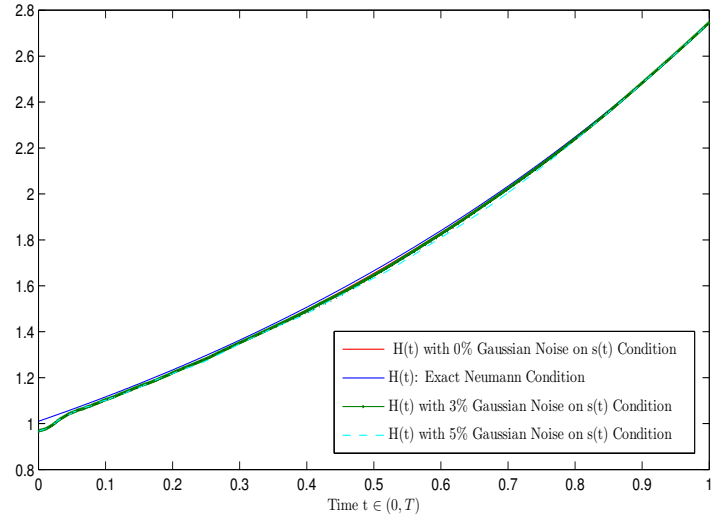}
\caption{The exact solution $-u_x(0, t)$ and approximate solution with gaussian noise obtained with $\lambda=1e-2$, $N=1000$.}
\label{xxx}
\end{figure}

\begin{table}[H]
\centering
\begin{tabular}{|c|c|c|}
\hline
 $\lambda$ & Noise on $s(t)$ ($\%$) &$\dfrac{\ds\norm{h_{exact} - h}_2}{\ds\norm{h_{exact}}_2}$ \\
\hline
 0.01 & 0 $\%$ & 0.0066 \\
\hline
 0.01 & 1 $\%$ & 0.0068 \\
 \hline
 0.01 & 3 $\%$ & 0.0080 \\
\hline
 0.01 & 5 $\%$ & 0.0102 \\
\hline
\end{tabular}
\caption{Relative errors for all cases.}
\label{Table_relative_error}
\end{table}

\begin{figure}[H]
\centering
\includegraphics[width=11.5cm,height=7.5cm]{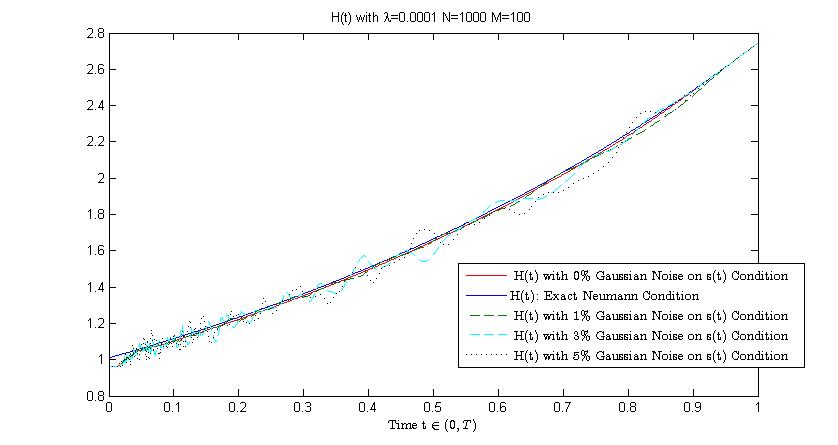}
\caption{The exact solution $-u_x(0, t)$ and the approximate one with different Gaussian noise levels obtained with $\lambda=10^{-4}$, $N=1000$.}
\label{yyy}
\end{figure}

\begin{table}[H]
\centering
\begin{tabular}{|c|c|c|}
\hline
 $\lambda$ & Noise on $s(t)$ ($\%$) &$\dfrac{\ds\norm{h_{exact} - h}_2}{\ds\norm{h_{exact}}_2}$ \\
\hline
 0.0001 & 0 $\%$ & 0.0081 \\
\hline
 0.0001 & 1 $\%$ & 0.0094 \\
 \hline
 0.0001 & 3 $\%$ & 0.0189 \\
\hline
 0.0001 & 5 $\%$ & 0.0292 \\
\hline
\end{tabular}
\caption{Relative errors for all cases.}
\label{Table_relative_error}
\end{table}

The Tikhonov regularization parameter $\lambda$ is chosen arbitrarily and the results in the figures and tables above show that we obtain a better precision with $\lambda$ of the order $10^{-2}$ which underscore the need of  regularization (\ref{Tikhonov}).

\paragraph{ {\bf Example 3}}
In this example the moving boundary is given by the nonlinear function 
\begin{equation}
s(t)=\sqrt{t+\frac{1}{4}},  \ t\in[0,1].
\end{equation}
and has the initial condition
\begin{equation}
u_0(x)=\dfrac{exp(\frac{1}{4}) \sqrt{\pi}} {2}   \big( erf(\frac{1}{2}) - erf(x)\big), \ x\in [0,\frac{1}{2}], 
\end{equation} 
where $erf(x)$ is the error function given by 
$erf(x)= \frac{2}{\pi} \ds\int^x_0 exp(-t^2) \ dt$.
This example has the Neumann boundary condition 
$h(t)=  \ds\dfrac{exp(\frac{1}{4})}{2\sqrt{t+\frac{1}{4}}}$,
which is taken from \cite{Kn, andreucci2004lecture}. Here, we considered a Stefan problem with a 
free boundary function's regularity  that is beyond  the framework of the derived stability estimates.

\begin{figure}[H]
\centering
\includegraphics[width=11.5cm,height=7.5cm]{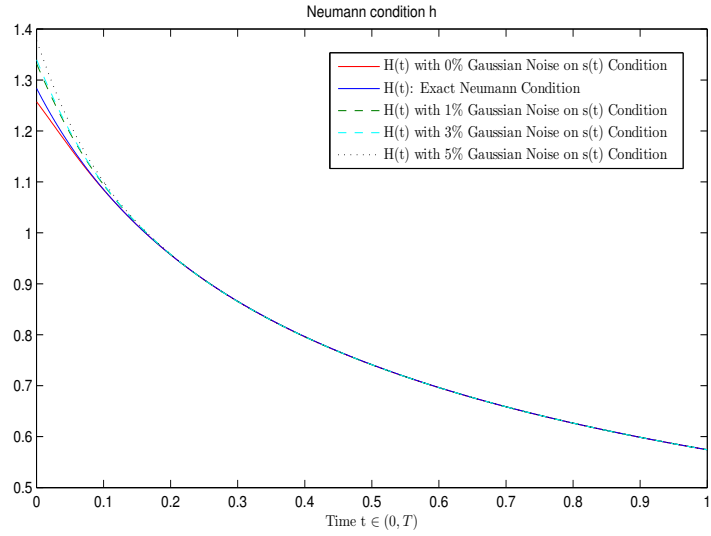}
\caption{The exact solution $-u_x(0, t)$ and approximate solution with different Gaussian noise levels obtained with $\lambda=10^{-3}$, $N=1000$.}
\label{HKnabner}
\end{figure}

\begin{figure}[H]
\centering
\includegraphics[width=11.5cm,height=7.5cm]{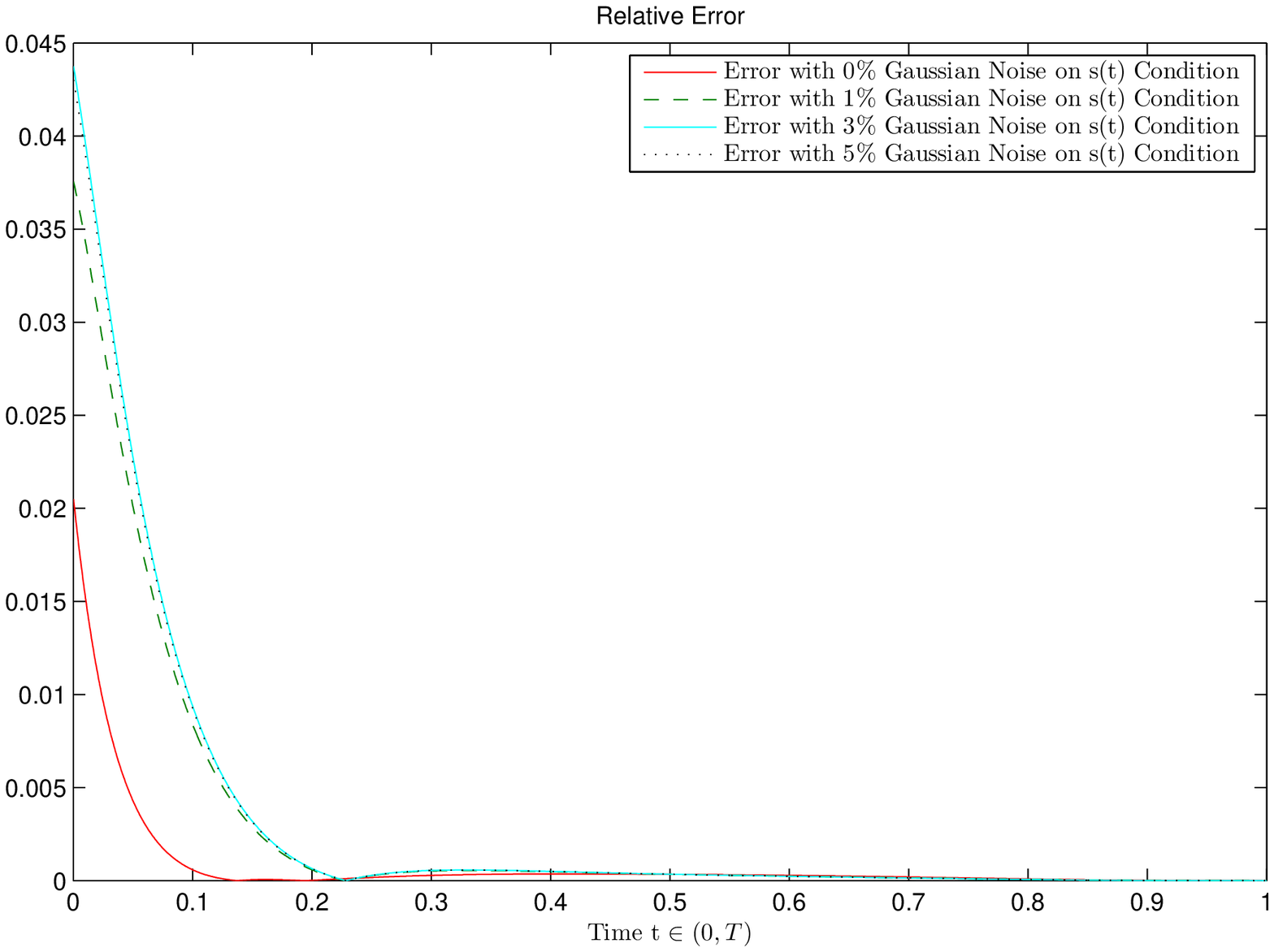}
\caption{The relative error with different Gaussian noise levels obtained with $\lambda=10^{-3}$, $N=1000$.}
\label{ErrorKnabner}
\end{figure}

\begin{table}[H]
\centering
\begin{tabular}{|c|c|c|}
\hline
 $\lambda$ & Noise on $s(t)$ ($\%$) &$\dfrac{\ds\norm{h_{exact} - h}_2}{\ds\norm{h_{exact}}_2}$ \\
\hline
 0.001 & 0 $\%$ & 0.0040 \\
\hline
 0.001 & 1 $\%$ & 0.0110 \\
 \hline
 0.001 & 3 $\%$ & 0.0126 \\
\hline
 0.001 & 5 $\%$ & 0.0124 \\
\hline
\end{tabular}
\caption{}{Relative errors for all cases.}
\label{Table_relative_error_Knanber}
\end{table}

In Figure (\ref{HKnabner}), we plot the Neumann condition $h$ for three different Gaussian noise levels $1$, $3$ and $5\%$ and when compared, they match well with the exact solution especially far from $t$ tending to $0$.

\begin{remark}
Notice  that in the three considered  examples  the explicit boundary influxes  are analytic functions of $t$. This explains the
 relatively good recovery of the boundary influx from the measurement of the free boundary far away
from $t=0$.
\end{remark}

\begin{remark}
 Recall that for a smooth right hand side $f$  the inverse of  the Abel integral operator is given by \cite{Br}
\bea
A_0^{-1} f(t) = \frac{1}{\sqrt \pi} \frac{f(0)}{\sqrt t} + A_0 f^\prime(t), \textrm{ for } t\in (0, T).
\eea
It seems that the regularization, and discretization   \eqref{ust}  of the Abel-like integral equation  \eqref{eqq}  produce a numerical 
weak singularity at the origin  that doesn't match the real regularity of  the boundary influx, and  
slightly increases the relative  error there. This numerical problem will be considered  in future
works.

\end{remark}


\end{document}